\documentclass[11pt]{amsart}

\usepackage[T1]{fontenc}

\usepackage{latexsym,enumerate}

\usepackage{amsmath,amssymb,amsthm,amsfonts,latexsym}
\usepackage{pstricks, pst-node, pst-text, pst-3d}
\usepackage[bookmarks]{hyperref}

\def\cal{\mathcal}

\def\adh#1{\overline{#1}}

\let\=\partial

\newtheorem {pro}{Proposition}[section]
\newtheorem {thm}[pro]{Theorem}
\newtheorem{lem}[pro]{Lemma}

\theoremstyle{definition}
 \newtheorem {rem}[pro]{Remark}
\newtheorem {dfn}[pro]{Definition}

\newtheorem {exa}[pro]{Example}

\newcommand{\hti}{\tilde{h}}

\newcommand{\bba}{\adh{B}}
\newcommand{\jac}{\mbox{Jac}}

\newcommand{\R}{\mathbb{R}}

\newcommand{\Bb}{ \overline{B}}

\newcommand{\ep}{\varepsilon}

\newcommand{\pa}{\partial}

\title[]{Poincar\'e inequality on subanalytic sets}

\makeatletter

\@namedef{subjclassname@2020}{%
  \textup{2020} Mathematics Subject Classification}

\makeatother

\author[A. Valette  and G. Valette]{Anna Valette  and Guillaume Valette}

\address[A. Valette]{Katedra Teorii Optymalizacji i Sterowania, Wydzia\l\ Matematyki i Informatyki Uniwersytetu Jagiello\'nskiego, ul. S. \L ojasiewicza 6, Krak\'ow, Poland}
\email{anna.valette@im.uj.edu.pl}\address[G. Valette]{Instytut Matematyki Uniwersytetu
Jagiello\'nskiego, ul. S. \L ojasiewicza 6, Krak\'ow, Poland}\email{guillaume.valette@im.uj.edu.pl}

\keywords{Poincar\'e inequality, subanalytic sets, singular domains, Sobolev spaces}

\thanks{Research partially supported by the NCN grant  2014/13/B/ST1/00543.}

\subjclass[2020]{26D10, 32B20, 14P10}

\begin{document}
\begin{abstract}
Let $\Omega$ be a subanalytic connected bounded open subset of $\R^n$,
with possibly singular boundary. We show that given $p\in
[1,\infty)$,  there is a constant $C$ such that for any  $u\in
W^{1,p}(\Omega)$ we have
$||u-u_{\Omega}||_{L^p} \le C||\nabla u||_{L^p},$
where we have set  $u_{\Omega}:=\frac{1}{|\Omega|}\int_{\Omega} u.$
\end{abstract}
\maketitle
\section{Introduction}There are several different types of inequalities being called by the name of the great French mathematician.
One of them asserts that, given a connected bounded open subset $\Omega$ of $\R^n$ with Lipschitz boundary and $p\in [1,\infty)$,   there is a constant $C$ such that for any  $u\in W^{1,p}(\Omega)$ we have 
$$||u-u_{\Omega}||_p \le C||\nabla u||_p,$$
where we have set  \begin{equation}\label{u}u_{\Omega}:=\frac{1}{|\Omega|}\int_{\Omega} u(x)\,dx \end{equation} and $|\Omega|$ stands for the Lebesgue measure of $\Omega$.

Under this form, it is sometimes called Poincar\'e-Wirtinger inequality. This inequality plays an important role in the theory of partial differential equations. It is well-known that it is no longer true if we drop the assumption that  $\Omega$ has Lipschitz boundary.
 It is actually an interesting problem to study the interplay between the geometry of the singularities  of the boundary and this result of analysis \cite{m} (see for example sections 1.1.11 and 6.4 of this book).

We show  in this article that this inequality holds on every subanalytic connected bounded open subset of $\R^n$, with possibly singular boundary. The idea is to use the techniques that the second author recently developed to study $L^p$ de Rham cohomology on singular subanalytic varieties \cite{gvpoincare} (see Theorem \ref{thm_local_conic_structure} below).  We do not put any extra {\em ad hoc} assumption on the Lipschitz geometry of the boundary.
  
The geometric properties that are needed are not really specific to the subanalytic category and could be proved to be valid on every polynomially bounded o-minimal structure \cite{c,vd}, with almost no modifications in the proofs, which means that our theorem could be established for domains that are definable in such a structure.
 It seems that the open sets that are definable in these structures constitute a natural class 
of singular domains to extend the theory of partial differential equations, and indeed the case of semi-algebraic domains, on which it is possible to carry out effective computations, is already satisfying for most of the applications.

We start by giving definitions and needed facts of subanalytic geometry.   Since we do not restrict ourselves to star-shaped domains, the proof of Poincar\'e  inequality will force us to construct homotopies that will not be smooth, yet will only be continuous almost everywhere. We thus establish two lemmas that are devoted to the construction of these homotopies, and then state and prove the main theorem (Theorem \ref{p2}). Although these techniques are very particular to our framework, these lemmas in fact bear some resemblance with tools developed to investigate much wider classes of metric spaces \cite{semmes}.

 We denote by
$||\cdot||$ the Euclidean norm of $\R^n$, by $B(x_0,\ep)$ (resp. $\overline{B}(x_0,\ep)$) the open (resp. closed) ball of radius $\ep>0$ centered at $x_0\in \R^n$ for this norm, and by $S(x_0,\ep)$  the corresponding sphere.
\section{Subanalytic sets}
 We now  recall some basic facts about subanalytic sets and functions. We refer to \cite{bm} (see also \cite{ds, l, livre}) for proofs and related facts.

\begin{dfn}\label{dfn_semianalytic}
A subset $E\subset \R^n$ is called {\bf semi-analytic} if it is  locally
defined by finitely many real analytic equalities and inequalities. Namely, for each $a \in   \R^n$, there is
a neighborhood $U$ of $a$, and real analytic  functions $f_{ij}, g_{ij}$ on $U$, where $i = 1, \dots, r, j = 1, \dots , s_i$, such that
\begin{equation}\label{eq_definition_semi}
E \cap   U = \bigcup _{i=1}^r\bigcap _{j=1} ^{s_i} \{x \in U : g_{ij}(x) > 0 \mbox{ and } f_{ij}(x) = 0\}.
\end{equation}
\end{dfn}

The flaw of the  semi-analytic category is that  it is not preserved by analytic morphisms, even when they are proper. To overcome this problem, it is convenient to work with a bigger class of sets, the  subanalytic sets, which are defined as the projections of the semi-analytic sets:

\medskip

\begin{dfn}
 A subset $E\subset \R^n$  is  {\bf  subanalytic} if 
 each point $x\in\R^n$ has a neighbourhood $U$ such that $U\cap E$ is the image under the canonical projection $\pi:\R^n\times\R^k\to\R^n$ of some relatively compact semi-analytic subset of $\R^n\times\R^k$ (where $k$ can depend on $x$).
 \end{dfn}

%
\begin{dfn}
   We say that {\bf a mapping $f:A \to B$ is  subanalytic}, $A \subset \R^n$, $B\subset \R^m$ subanalytic, if its graph is a  subanalytic subset of $\R^{n+m}$. In the case $B=\R$, we say that  $f$ is a {\bf  subanalytic function}.
\end{dfn}
Subanalytic sets constitute a nice category to study the geometry of semi-analytic sets: it is  stable under union, intersection, complement, and Cartesian product. The closure of a subanalytic set is subanalytic. Moreover, these sets enjoy many finiteness properties. For instance, bounded subanalytic sets always have finitely many connected components, each of them being   subanalytic. 

It is also well-known that they are $C^0$ triangulable, in the sense that a subanalytic set is always homeomorphic to a simplicial complex. This implies in particular that locally, they have the topology of a cone over what is generally called, its link. This fact is sometimes referred as the {\it local conic structure} of the topology of subanalytic sets. Germs of subanalytic sets are nevertheless not bi-Lipschitz homeomorphic to cones, as it is shown by the simple example of a cusp $y^2=x^3$ in $\R^2$.  In particular, the cone property which is often used in functional analysis (see \cite{m} for the  definition), which is of metric nature, may fail.  The theorem below, achieved in \cite{gvpoincare}, however unravels the Lipschitz properties of the local conic structure. This Lipschitz conic structure was actually derived from techniques developed to construct triangulations that describe the metric properties of singularities \cite{vlt} (see also the survey \cite{livre}).


In the theorem below, $x_0* (S(x_0,\ep)\cap X)$ stands for the cone over $S(x_0,\ep)\cap X$ with vertex at $x_0$.

\begin{thm}[Lipschitz Conic Structure]\label{thm_local_conic_structure}
  Let  $X\subset \R^n$ be subanalytic and $x_0\in X $. 
For $\ep>0$ small enough, there exists a Lipschitz subanalytic homeomorphism
$$H: x_0* (S(x_0,\ep)\cap X)\to  \Bb(x_0,\ep) \cap X,$$  
  satisfying $H_{| S(x_0,\ep)\cap X}=Id$, preserving the distance to $x_0$, and having the following metric properties:
\begin{enumerate}[(i)] 
 \item\label{item_H_bi}     The natural retraction by deformation onto $x_0$ $$r:[0,1]\times  \Bb(x_0,\ep)\cap X \to \Bb(x_0,\ep)\cap X,$$ defined by $$r(s,x):=H(sH^{-1}(x)+(1-s)x_0),$$ is Lipschitz.   
 Indeed, there is a constant $C$ such that  for every fixed $s\in [0,1]$, the mapping $r_s$ defined by $x\mapsto r_s(x):=r(s,x)$, is $Cs$-Lipschitz.
 \item \label{item_r_bi}  For each $\delta>0$,
 the restriction of $H^{-1}$ to $\{x\in X:\delta \le ||x-x_0||\le \ep\}$ is Lipschitz and, for each $s\in (0,1]$, the map  $r_s^{-1}:\Bb(x_0,s\ep) \cap X\to \Bb(x_0,\ep) \cap X$ is Lipschitz. 
\end{enumerate}
\end{thm}
\begin{rem}\label{rem}
This theorem is actually valid in every polynomially bounded o-minimal structure, the same proof applying, simply replacing the word ``subanalytic'' with definable.
\end{rem}

\begin{exa}The difference between the notions of Lipschitz conic structure of a set and cone property may be unclear to the reader, and we wish to illustrate it by sketching what can be the mapping $H$ of the above theorem on the example of a cusp $$X:=\{(x,y)\in [0,1]\times \R:|y|\le x^2\}$$ with $x_0=(0,0)$. For each $(x,y)\in x_0*(S(0,1)\cap X)$, let  $$H(x,y):=(t(x,y) x,\,t^2(x,y)xy),$$ where 
$$t(x,y)=\left(\frac{2x^2+2y^2}{x^2+\sqrt{x^4+4x^2y^2(x^2+y^2)}}\right)^{1/2}.$$
The choice that we made for $t(x,y)$ ensures that this mapping preserves the distance to the origin. Moreover, a straightforward computation yields that on $x_0*(S(0,1)\cap X)$ we have $|\pa t(x,y)|\le \frac{C}{x}$ for some positive constant $C$, from which it comes down that $H$ has bounded first order partial derivatives (the function $t$ is bounded away from zero and infinity on $x_0*(S(0,1)\cap X)$).

Let us here emphasize that if we drop the condition that $H$ preserves the distance to the origin but simply require $\frac{|(x,y)|}{C}\le  |H(x,y)| \le C|(x,y)| $ for some constant $C$ (which is sufficient for proving Poincar\'e inequality) then it suffices to set $H(x,y)=(x,xy)$,  which leads to much easier computations.
\end{exa}

\section{Poincar\'e inequality}
For an open subset $\Omega\subset\R^n$ and $p\ge 1$ we denote by $$W^{1,p}(\Omega):= \{u\in L^p(\Omega),\; \frac{\partial u}{\pa x_i}\in L^p(\Omega)\; \mbox{\rm for any }\, 1\le i\le n\},$$ the Sobolev space, where $\frac{\partial u}{\partial x_i}$ are the partial derivatives of $u$ in the sense of distributions.
This space, equipped with the norm
$$||u||_{p}+\sum_{i=1}^n||\frac{\partial u}{\pa x_i}||_{p},$$ is a Banach space. Here, as usual, $||\cdot||_p$ stands for the $L^p$ norm.
It is well known that the set of smooth functions $\cal C^\infty(\Omega)$ is dense in  $W^{1,p}(\Omega)$ .


The proof of  Theorem \ref{p2} will require two geometric lemmas that we now present, the first being necessary to establish the second one. Let us recall that since subanalytic mappings are differentiable on an open dense subanalytic subset of their domain, they are always differentiable almost everywhere.

\begin{lem}\label{lem_1}
Let $\Omega\subset\R^n$ be a bounded open connected   subanalytic subset. There exists a  subanalytic map 
$$h: \Omega\times[0,1]\to \Omega, \, (x,s)\mapsto h_s(x)$$ continuous with respect to the second variable  and such that
\begin{enumerate}
\item $h_1(\Omega)\subset B(z,\alpha)\subset\Omega$ for some $\alpha>0$ and $z\in \Omega$;
\item $d_xh_t$ is invertible for almost every $(x,t)\in\Omega\times[0,1]$, and moreover there exists $C>0$ such that whenever  $d_xh_t$ is invertible,
we have $||d_xh_t^{-1}||\le C$.\end{enumerate}
\end{lem}
\begin{proof}
The proof relies on the following two observations. 

\begin{enumerate}[(a)]
\item  In condition (1), possibly composing $h$ with a homothetic transformation, we can choose $\alpha$ as small as we wish. Moreover,  as $\Omega$ is connected and  subanalytic, any two points of $\Omega$ can be joined by a  subanalytic arc  $\gamma$.  Therefore, possibly composing with a translation of the ball $B(z,\varepsilon)$ through an arc $\gamma$, the point $z$ can be replaced with any element of $\Omega$.    
\item\label{union} It is enough to construct the desired family of maps on a finite (subanalytic) cover  of $\Omega$. The reason is that, if there exist mappings $h:U \times[0,1]\to \Omega$  and $h':U'\times[0,1]\to\Omega$, where $U, U'\subset\Omega$ are subanalytic subsets, satisfying $(1)$ and $(2)$ of the lemma, then we can construct a subanalytic mapping $h'':(U\cup U')\times[0,1]\to\Omega$ satisfying $(1)$ and $(2)$ as well. Indeed, for instance, it is enough to set $$h''(x,t)=\begin{cases}h(x,t),&  x\in U\\
h'(x,t),& x\in U'\setminus U.\end{cases}$$
The obtained mapping $h''$ is then subanalytic and continuous with respect to $t$ (it may fail to be continuous with respect to $x$ but, as it is subanalytic, it is smooth on an open dense subanalytic set). By induction, given several sets $U_1,\dots,U_k$ and several mappings $h_i:U_i\times[0,1]\to\Omega$, $i=1,\dots,k$, we therefore can define a mapping $h:\bigcup_{i=1}^k U_i \times[0,1]\to\Omega$ that has the required properties.
\end{enumerate}

As the set $\adh\Omega$ is compact, by the just above observation (\ref{union}), it suffices to define $h_s(x)$ on $\bba(x_0,\ep)\cap \Omega$, for each $x_0\in \adh\Omega$, with $\ep>0$ small. Of course, the result is clear if $x_0\in \Omega$.  \\ 
Fix now $x_0$ in the boundary of $\Omega$.   Given $s\in [0,1)$, define a function $\nu_s:[0,1]\to [s,1]$ by $$\nu_s(u)=s+(1-s)u.$$  
For every $s<1$,  $\nu_s'(u)\equiv 1-s$, and therefore $\nu_s$  is a bi-Lipschitz homeomorphism with bi-Lipschitz constant $L_{\nu_s}$ which remains bounded if $s$ stays  bounded  away  from $1$. Observe also that, since we can argue separately on each connected component of $\bba(x_0,\ep)\cap \Omega$ (thanks to (\ref{union})), it is no loss of generality to assume that this set is connected.

Let  $r$ and $\ep$ be as in the Lipschitz Conic Structure Theorem  (applied to $\Omega\cup \{x_0\}$ at $x_0$).  Set $\rho_\ep(x):=\frac{||x-x_0||}{\ep}$.
We claim that for every $s\in[0,1)$ the mapping $x \mapsto r_{\nu_s^{-1}(\rho_\ep(x))}(x)$ induces a homeomorphism $\theta_s$ from  $\adh{B}(x_0,\ep)\cap \Omega \setminus \bba(x_0,s\ep)$ onto $\adh{B}(x_0,\ep)\cap \Omega $.  Actually, since the homeomorphism $H$, given by the Lipschitz Conic Structure Theorem, preserves the distance to the point $x_0$, and because for each $s$ the map $r_s$ is,  up to $H$, a homothetic transformation (i.e., $H\circ r_s\circ H^{-1}$ is a homothetic transformation),  $r_s$ maps injectively $S(x_0,t)\cap \Omega$ onto $S(x_0,ts)\cap \Omega$ for all $s\in (0,1]$ and $t\in [0,\ep]$.  As $\nu_s$ is a strictly increasing function, it therefore suffices to check that $x\mapsto r_{\nu_s^{-1}(\rho_\ep(x))}(x)$ maps $S(x_0,\ep)\cap \Omega$ onto itself and $S(x_0,s\ep)\cap \Omega$ onto $\{x_0\}$, that is to say, that $\nu_s(1)=1$ and  $\nu_s(0)=s$, which is clear from the definition of $\nu_s$.

Hence, we can define for each $s\in[0,\frac{1}{2}]$ a mapping 
\begin{equation}g_s:\adh{B}(x_0,\ep)\cap \Omega \to \adh{B}(x_0,\ep)\cap \Omega \setminus \bba(x_0,s\ep)\end{equation} by setting  for $y\in \adh{B}(x_0,\varepsilon)\cap \Omega$
$$g_s(y):=\theta^{-1}_s(y).$$
Note that, since $g^{-1}_s=\theta_s$, it is a Lipschitz mapping, with a Lipschitz constant which can be bounded independently of $s\le\frac{1}{2}$. 

By (\ref{union}) we may assume that $S(x_0,\ep)\cap \Omega$ is included in a chart of some coordinate system of $\Omega$. By induction on $n$, we therefore know that there is a family of mappings $$\hti_s:S(x_0,\ep) \cap \Omega\to S(x_0,\ep) \cap \Omega, \quad s\in [0,1] $$ satisfying  $\hti_1(S(x_0,\ep) \cap \Omega)\subset B(a,\tilde{\alpha})$ for some $a\in S(x_0,\ep)\cap\Omega$ and $\tilde{\alpha}$ small enough, such that $d \hti_s^{-1}$  is bounded uniformly in $s$.  Let us extend trivially (i.e., constantly with respect to the last variable)  this family of mappings   to a family of mappings, keeping the same notation $$\hti_s:S(x_0,\ep) \cap \Omega \times [\frac{1}{2},1]\to S(x_0,\ep) \cap \Omega  \times [\frac{1}{2},1], \quad s\in [\frac{1}{2},1] .$$

We now shall define the desired mapping $h$ by applying successively $g$ and $\hti$. 
Observe for this purpose that $r$ induces a bi-Lipschitz homeomorphisms from $S(x_0,\ep)\cap \Omega \times [\frac{1}{2},1]$ to $ \Omega \cap \bba(x_0,\ep)\setminus B(x_0,\frac{\ep}{2})$. Denote by $\Psi$ its inverse and let $h_s:\Omega\cap \bba(x_0,\ep)\to \Omega\cap \bba(x_0,\ep) $ be defined by 
$$h_s(x)=\begin{cases}g_s(x), & s\le \frac{1}{2} \\
r(\hti_s(\Psi(g_{1/2}(x))), &s\ge \frac{1}{2} 
\end{cases}.
$$
The inverse of its derivative is bounded by construction.
\end{proof}

The above lemma makes it possible for us to prove the following lemma which will be useful to establish Theorem \ref{p2}.
 We use the notation $\jac(\Gamma)$ to denote the absolute value of the determinant of the Jacobian matrix of a mapping $\Gamma$.
\begin{lem}\label{lem_2}
Let $\Omega\subset\R^n$ be an open bounded connected subanalytic subset.  There exists a subanalytic  family of  continuous arcs 
$\gamma_{x,y}:[0,1]\to \Omega$, $x,y\in\Omega$, such that $\gamma(0)=x$, $\gamma(1)=y$ for each such $x$ and $y$, and $||\gamma'_{x,y}(s)||\le C$ for all $s\in(0,1)$, and some constant $C$ independent of $x$ and $y$. Moreover, there is $\eta>0$ such that for almost every $x,y\in\Omega$:
$$ \begin{cases}\jac(\Gamma_{s,y})(x)\ge \eta,& s\ge \frac{1}{2}\\
\jac(\tilde{\Gamma}_{s,x})(y)\ge \eta,& s< \frac{1}{2}\end{cases}$$ where 
 $\Gamma_{s,y}:\Omega\ni x\mapsto\gamma_{x,y}(s)\in\Omega$
 and $\tilde{\Gamma}_{s,x}:\Omega\ni y\mapsto\gamma_{x,y}(s)\in\Omega$.
\end{lem}
\begin{proof}

The desired family of arcs will require the following mappings. Let:
\begin{enumerate}
\item 
 $h:\Omega\times[0,1]\to \Omega$ be a family of mappings as provided by Lemma \ref{lem_1}.
 \item $g:B(z,\alpha)\times B(z,\alpha)\times[0,1]\to B(z,\alpha)$ be defined by $(x,y,s)\mapsto (1-s)x+sy$ (where $z$ and $\alpha$ are provided by Lemma \ref{lem_1}).
\end{enumerate}
Next, we define $H:\Omega\times \Omega \times[0,3]\to\Omega$ by
$$H(x,y,s)=\begin{cases}h(x,s)& 0\le s\le 1\\
g(h(x,1),h(y,1),s-1)& 1<s\le 2\\
h(y,3-s)& 2<s\le 3.
\end{cases}
$$
 Note that, by $(1)$ of Lemma \ref{lem_1}, we know that there is $\lambda$ such that $|\det d_x h_s^{-1}|\le \lambda$ for almost every $x$ in $\Omega$. As a matter of fact:
\begin{itemize}
\item for $s\le \frac{3}{2}$ and any fixed $y$, the jacobian of the map $H_{s,y}:x\mapsto H(x,y,s)$ can be bounded as follows:
$$\jac (H_{s,y})(x)=\begin{cases}\jac (h_s)(x)=\frac{1}{|\det d_x h_s^{-1}| }\ge \frac{1}{\lambda}& s\in[0,1]\\
(2-s)^n\jac (h_{1})(x)>\frac{1}{2^n\lambda}& s\in(1,\frac{3}{2}];\end{cases}$$
\item  for $s> \frac{3}{2}$ and any fixed $x$ we have:
 $$ \jac (H_{s,x})(y)=\begin{cases}(s-1)^n \jac (h_{1})(y)>\frac{1}{2^n\lambda}& s\in(\frac{3}{2},2]\\
\jac (h_{3-s})(y)=\frac{1}{|\det d_y h_{3-s}^{-1}| }\ge \frac{1}{\lambda}&s\in(2,3].\end{cases}$$

\end{itemize}
Hence, the family of  arcs  $\gamma_{x,y}:[0,1]\to\Omega$ 
defined by $\gamma_{x,y}(s)=H(x,y,3s)$ fulfills  the required  properties.
\end{proof}

We are now ready to prove our main result:
\begin{thm}\label{p2}
Let $\Omega\subset\R^n$ be a bounded connected open subanalytic subset. For each $p\ge 1$, there exists $C>0$  such that for any $u\in W^{1,p}(\Omega)$ the following inequality holds
$$||u-u_{\Omega}||_p \le C||\nabla u||_p,$$ where $u_\Omega$ is as in (\ref{u}).
\end{thm}
 \begin{proof}We may assume that $u$ is smooth. For $x$ and $y$ in $\Omega$, let $\gamma_{x,y}$ be the mapping provided by  Lemma \ref{lem_2}. 
We have
\begin{eqnarray}\label{eq_lp}||u-u_{\Omega}||_p\nonumber&=&\left(\int_{\Omega}\left|u(x)-\frac{1}{|\Omega|}\int_{\Omega} u(y)dy\right|^pdx\right)^{1/p} \\ \nonumber
&=& \left(\int_{\Omega}\frac{1}{|\Omega|^p}\left|\int_{\Omega} u(x)-u(y)\,dy\right|^pdx\right)^{1/p} \\ \nonumber
&=& \frac{1}{|\Omega|}\left(\int_{\Omega}\left|\int_{\Omega}\int_0^1<\nabla u (\gamma_{x,y}(s)),\gamma'_{x,y}(s)>dsdy\right|^pdx\right)^{1/p}\\ \nonumber
&\le& \frac{1}{|\Omega|}\left(\int_{\Omega}\left|\int_0^1\int_{\Omega}||\nabla u (\gamma_{x,y}(s))||\cdot||\gamma'_{x,y}(s)||dyds\right|^pdx \right)^{1/p}\\
&\le& \frac{C}{|\Omega|}\left(\int_{\Omega}\left|\int_0^1\int_{\Omega}||\nabla u (\gamma_{x,y}(s))||dyds\right|^pdx \right)^{1/p}.
\end{eqnarray}
Thanks to Minkowski's inequality, we can write
\begin{eqnarray*} \left(\int_{\Omega}\left|\int_{1/2}^1\int_{\Omega}||\nabla u (\gamma_{x,y}(s))||dyds\right|^pdx \right)^{1/p}&\le& \int_{1/2}^1\int_{\Omega}\left(\int_{\Omega}||\nabla u (\gamma_{x,y}(s))||^pdx\right)^{1/p}dyds.
\end{eqnarray*}

Note that, by Lemma \ref{lem_2},  we have for all $s\in [1/2,1]$:
$$\int_{\Omega}||\nabla u (\gamma_{x,y}(s))||^p dx
\le \frac{1}{\eta}\int_{\Omega}||\nabla u (\gamma_{x,y}(s))||^p \jac(\Gamma_{s,y}(x))dx \le C\frac{||\nabla u||_p^p}{\eta},$$
for some constant $C$, so that, plugging this into the preceding estimate, we get:
\begin{equation}\label{eq_s_le}
 \left(\int_{\Omega}\left|\int_{1/2}^1\int_{\Omega}||\nabla u (\gamma_{x,y}(s))||dyds\right|^pdx \right)^{1/p}\le C\frac{|\Omega|}{\eta^{1/p}}||\nabla u||_p.
\end{equation}
Moreover, denoting by $p'$ the H\"older conjugate of $p$, thanks to H\"older's inequality, we can write: 
$$\left|\int_{\Omega}\int_0^{1/2}||\nabla u (\gamma_{x,y}(s))||dsdy\right|^p\le \frac{|\Omega|^{p/p'}}{2^{p/p'}} \int_{\Omega}\int_0^{1/2}||\nabla u (\gamma_{x,y}(s))||^p ds dy$$
$$\le \frac{|\Omega|^{p/p'}}{\eta\, 2^{p/p'}} \int_0^{1/2} \int_{\Omega}||\nabla u (\gamma_{x,y}(s))||^p \jac(\tilde{\Gamma}_{s,x}(y))dxds \le C \frac{|\Omega|^{p/p'}}{\eta \,2^{p/p'}} ||\nabla u||_p^p\,,$$for some constant $C$. 
Together with (\ref{eq_lp}) and (\ref{eq_s_le}), this yields the desired result.
\end{proof}

\end{document}